\documentclass[12pt,a4paper,leqno]{amsart}
\usepackage[utf8]{inputenc}
\usepackage{amssymb,color,tikz, cancel}
\usepackage{amsmath,amscd,amsfonts,amsthm,verbatim}
\usepackage{tikz-cd}
\usepackage{enumitem}
\usepackage[all]{xy}

\newcommand{\cE}{{\mathcal E}}

\newcommand{\cP}{{\mathcal P}}

\newcommand{\cX}{{\mathcal X}}

\newcommand{\cO}{{\mathcal O}}
\newcommand{\cI}{{\mathcal I}}
\newcommand{\cM}{{\mathcal M}}

\newcommand {\PP}{\mathbb{P}}
\newcommand{\cH}{{\mathcal H}}

\newcommand{\deb}{\overline{\partial}}
\newcommand{\K}{K}
\newcommand{\eps}{\varepsilon}

\newcommand {\QQ}{\mathbb{Q}}

\DeclareMathOperator{\EExt}{{\cE}xt}
\DeclareMathOperator{\Hilb}{Hilb}
\DeclareMathOperator{\Hom}{Hom}

\DeclareMathOperator{\Def}{Def}
\DeclareMathOperator{\at}{at}
\DeclareMathOperator{\Art}{{\mathcal{A}}rt}
\DeclareMathOperator{\Set}{{\mathcal{S}}et}

\DeclareMathOperator{\id}{id}

\DeclareMathOperator{\Ext}{Ext}

\DeclareMathOperator{\Spec}{Spec}

\DeclareMathOperator{\coker}{coker}
\DeclareMathOperator{\pt}{pt}

\newtheorem{theorem}{Theorem}[section]
\newtheorem{lemma}[theorem]{Lemma}

\newtheorem{proposition}[theorem]{Proposition}
\newtheorem{corollary}[theorem]{Corollary}
 \theoremstyle{definition}
\newtheorem{definition}[theorem]{Definition}

\newtheorem{remark}[theorem]{Remark}
\newtheorem{Notation}[theorem]{Notation}

\newcommand{\red}[1]{{\color{red} \sf $\clubsuit$ [#1]}}

\title{Smoothings of lci proper schemes}

\date{\today}

\begin{document}

 \author[B.\ Fantechi]{Barbara Fantechi}
 \address{SISSA
Via Bonomea 265, I-34136 Trieste, Italy}
  \email{fantechi@sissa.it,
  ORCID 0000-0002-7109-6818}.
  \author[R.\ M.\ Mir\'o-Roig]{Rosa M.\ Mir\'o-Roig}
  \address{Facultat de
  Matem\`atiques i Inform\`atica, Universitat de Barcelona, Gran Via des les
  Corts Catalanes 585, 08007 Barcelona, Spain} \email{miro@ub.edu, ORCID 0000-0003-1375-6547}

\thanks{The first author has been partially supported by PRIN 2022BTA242 “Geometry of algebraic structures: moduli, invariants,
deformations” and INdAM-GNSAGA}

\thanks{The second author has been partially supported by the grant PID2020-113674GB-I00}

\date{\today}

\begin{abstract}
We give criteria for the existence of geometric smoothings of a proper lci scheme or a DM stack $X$ as well as for a polarized lci scheme $(X,L)$, without assuming that $X$ is reduced. As applications, we give criteria for the smoothability of polarized K3 surfaces and of stable varieties.
\end{abstract}

\maketitle

\tableofcontents

\section{Introduction}

In a recent work, the authors established a criterion to ensure the smoothability of a polarized K3 surface $(X,L)$ with isolated  singularities \cite[Theorem 5.12]{FM}. This paper is motivated by our interest in extending this result to polarized schemes with singularities that are not necessarily isolated, and even to polarized schemes that are not necessarily reduced.

\begin{definition}

A {\em geometric smoothing} of $X$ is a Cartesian diagram
\[
\xymatrix{ X \ar[r]^{i} \ar[d] & \mathcal{X} \ar[d]^{\pi }\\
c\ar[r] & C
}
\]
where $C$ is a smooth irreducible curve, $c\in C$ is a closed point and $\pi $ is
a flat and proper morphism, such that $\pi $ is generically smooth. We say
that $X$ is {\em geometrically smoothable} if it has a geometric smoothing.
\end{definition}

The existence of geometric smoothings has been established in the literature for varieties with isolated singularities, see for instance \cite[Proposition 6.4]{Wa} and  \cite[Lemma 1]{Ma}.
For non-isolated singularities, there is a necessary criterion, the nonvanishing of the space of global sections of the sheaf  $T^1_X:= \EExt^1(L_X,\cO_{X})$ where $L_X$ is the cotangent complex of $X$ (note that if $X$ is reduced, the natural map $L_X\to \Omega_X$ is an isomorphism), but sufficient conditions are difficult to obtain.

A key role in studying the infinitesimal deformations of $X$, which will also be crucial in our approach, is the long exact sequence \begin{equation}
\label{long_ex_seq}
0\to H^1(T_X)\to \Ext^1(L_X,\cO_X)\to H^0(T^1_X)\to
\end{equation}
$$H^2(T_X)\to \Ext^2(L_X,\cO_X)\to H^1(T^1_X).
$$

Here $\Ext^1(L_X,\cO_X)$ is the vector space of first order deformations, $H^1(T_X)$ the same for locally trivial deformations. Therefore, it follows that $H^0(T^1_X)=0$ implies that all deformations are locally trivial, which means that $X$ is not geometrically smoothable.

\vskip 2mm

Our main result, proven in Section \ref{Proof_of_main_thm}, is the following theorem:

\begin{theorem}\label{main_theorem}
    Given a flat proper family $\pi:X_M\to M$ with $M$ irreducible, if there is $m_0\in M$ such that \begin{enumerate}
    \item $M$ is smooth at $m_0$,
    \item the fiber $X:=X_{m_0}:=\pi^{-1}(m_0)$ has lci singularities,
    \item the induced morphism $T_{m_0}M\otimes \cO_X\to T^1_X$ is surjective;
\end{enumerate}
then the open subset $U\subset M$ such that $\pi:\pi^{-1}(U)\to U$ is smooth is nonempty. In other words, the general fibre $X_m$ is smooth and $X$ admits a geometric smoothing obtained as base change from $\pi:X_M\to M$.
\end{theorem}

The morphism in (3) is obtained by composing the Kodaira Spencer map $T_{m_0}M\to \Ext^1(L_X,\cO_X)$ with the map $\Ext^1(L_X,\cO_X)\to H^0(T^1_X)$ from the exact sequence (\ref{long_ex_seq}), then tensoring with $\cO_X$ and composing with the evaluation morphism $H^0(T^1_X)\otimes \cO_X\to T^1_X$. In the proof of step (3) of Lemma \ref{key_lemma} we give a more direct construction of this map when $X$ is lci.

We also prove two generalizations of Theorem \ref{main_theorem}, namely that one can weaken assumption (3) to estimate the dimension of the singular locus in a generic fiber in Proposition \ref{singularity_estimate}, and that the theorem as stated applies also when $M$, $X_M$, or both are algebraic stacks in the sense of Deligne and Mumford in Proposition \ref{main_thm_stacks}.

The remainder of this paper focuses on applications of the main results.

In Section \ref{Geometric_smoothability_criteria}, we get sufficient conditions for the existence of a geometric smoothing for proper lci schemes and DM stacks, Theorem \ref{smoothable_Fano_gentype}, and for polarized projective lci schemes, Theorem \ref{smoothing_polarized}.

In particular, we prove the following.

\begin{theorem}\label{smoothability_criteria} Let $X$ be a proper lci pure dimensional scheme.
Assume  $H^2(X,T_X)=H^1(X,T^1_X)=0$, and $T^1_X$ is generated by global sections. If $\omega _X$ is ample or $\omega _X^{\vee}$ is ample or $H^2(X,\cO_X)=0$,  then $X$ admits a geometric smoothing.
\end{theorem}

In subsection \ref{K3_polar_smoothing}, we give a criterion for smoothability  of a polarized singular K3 surface $(X,L)$  in the sense of \cite{FM}. This  criterion extends \cite[Theorem 5.12]{FM} to the case where $X$ has non-isolated singularities, and  even to the case where $X$ is non-reduced; and gives a broader  criterion even in the case of isolated singularities.

Finally, in Section \ref{smoothing_stable_surfaces}, we give a criterion for the geometric smoothability of a stable variety $S$ whose associated index cover stack $X$ is lci.

\vskip 4mm
The present work is inspired by Tziolas' breakthrough results in \cite{Tz}, and our assumptions are similar to his. In \cite[Definition 11.6]{Tz}, he introduces the notion of formal smoothing, which is weaker than geometric smoothing. He then shows that for a proper equidimensional reduced lci  scheme $X$ such that $H^2(X,T_X)=H^1(X,T^1_X)=0$ and $T^1_X$ is generated by global sections,  a formal smoothing always exists \cite[Theorem 12.5]{Tz}.

Tziolas' argument is different from ours and based on studying deformations for formal schemes. In fact, his paper contains many results of interest on a functor we do not study at all, the deformations of the formal completion of such schemes along the singular locus. Moreover, he only uses the properness assumption on $X$ to show that the functor of infinitesimal deformations has a hull, which is true in more general cases.

Later, Nobile gave a characterization of formal smoothing for lci reduced schemes  in \cite[Theorem 5.11]{N} which only depends on a finite order deformation and not the whole formal smoothing. He then used it to show that, if $X$ is an lci reduced proper scheme such that $\omega _X$ is ample or $\omega _X^{\vee}$ is ample or $H^2(X,\cO_X)=0$, then the existence of a formal smoothing does imply that of a geometric smoothing \cite[Theorem 5.7]{N}.

Note that our result has extra assumptions when compared with Tziolas'. On the other hand, it also applies when $X$ is nonreduced, thus extending Nobile's result. We remark that our proof is independent of Tziolas' work and in particular does not use formal schemes.

\vskip 4mm

\noindent  {\bf Conventions.} Throughout this paper all schemes and algebraic stacks are of finite type over a fixed algebraically closed field $\K$ of characteristic 0. A point will always be a closed point.

\vskip 4mm

\noindent  {\bf Acknowledgment.}
The first author would like to thank the Universitat de Barcelona for hospitality while part of this paper was written. The authors would like to thank Rita Pardini for useful conversations.

\section{Preliminaries}

In this section, we gather some basic results that we will need later
on and we fix relevant notation.

\subsection{Cotangent complex and lci morphisms}

We recall here for the reader's convenience the properties of the cotangent complex we will use in this paper.

For a morphism $p:Y\to B$ of schemes or DM stacks we denote by $L_p\in D^{\le 0}_{coh}(Y)$ the cotangent complex of $p$ and by $T^1_p:=\EExt^1(L_p,\cO_{Y})$.

Given a cartesian diagram
\[
\begin{tikzcd}
Y' \arrow{r}{f} \arrow{d}{q} & Y\arrow{d}{p}\\
B'\arrow{r}{g} & B
\end{tikzcd}
\]
there is an induced morphism $\mathbb{L}f^*L_p\to L_q$, which is an equivalence if either $p$ or $g$ are flat.

Given morphisms $f:X\to Y$ and $g:Y\to Z$, we have a distinguished triangle on $X$\[
\begin{tikzcd}
    \mathbb{L}f^*L_g\arrow{r} & L_{g\circ f} \arrow{r} & L_f \arrow{r} & \mathbb{L}f^*L_g[1].
\end{tikzcd}
\]

\begin{definition}
  A morphism $f:Y\to B$ is lci (\lq locally complete
intersection\rq)  if it locally factors as a (closed) regular embedding followed by a smooth morphism. This is equivalent to the cotangent complex $L_f$ being perfect in $[-1,0]$.
\end{definition}

\begin{remark}  If $f:Y\to B$ is flat and lci, then $T^1_f$ commutes with base change, see e.g.\ \cite[Theorem 2.11]{FFP}.
\end{remark}

Let $f:X\to Y$ be a morphism of schemes and $E$ a perfect complex on $X$. Then there is a well defined morphism in $D^-(X)$
\[\at_E:E\otimes E^\vee[-1]\to L_f\]
called the {\em Atiyah class} with good functorial properties, see for instance \cite[Section 1]{Kn}.

\subsection{Deformation functors}

\smallskip
We denote by $\Art$ the category of local Artinian $\K$-algebras $A$ with residue field $\K$, and by $\Set$ the category of sets.
In this paper, a {\em deformation functor} is a functor $F:\Art\to\Set$ such that $F(\K)$ is one point and which admits tangent space $T^1F$ and obstruction space $T^2F$ in the sense of \cite{FMa}.

The basic example is the functor $h_{X,x}$ associated to a point $x$ in $X$ a scheme or a DM stack. It is defined by\[
h_{X,x}(A):=\Hom(\cO_{X,x},A)=\Hom(\widehat\cO_{X,x},A).
\]
We have $T^1h_{X,x}=T_xX$ and $T^2h_{X,x}=T^1_X(x):=T^1_{X,x}\otimes \kappa(x)$.

In the special case where $X=\Spec \K$ and $x\in X$ is the unique point, we denote by $\pt$ the associated deformation functor. For every other deformation functor $F$, there are unique natural transformations $\pt\to F\to \pt$.

We list here the deformation functors we use in this paper, indicating tangent and obstruction spaces. We let $X$ be a proper scheme or DM stack, $E$ a vector bundle on $X$, and $s_1,\ldots,s_n$ global sections of $E$.

We define a deformation of $X$ over $A\in \Art$ to be a triple $(X_A,i,p)$ forming a cartesian diagram \[
\begin{tikzcd}
X \arrow{r}{i} \arrow{d}{} & X_A\arrow{d}{p}\\
\Spec \K\arrow{r}{} & \Spec A
\end{tikzcd}
\]
An isomorphism $\phi:(X_A,i,p)\to(X'_A,i',p')$ is an isomorphism $$\phi:X_A\to X'_A$$ such that $i'=\phi\circ i$ and $p=p'\circ \phi$. The deformation functor $\Def_X$ is defined by $\Def_X(A):=$ the set of isomorphism classes of deformations of $X$ over $A$, and $T^i\Def_X=\Ext^i(L_X,\cO_X)$.

We define a deformation of $(X,E)$ over $A\in \Art$ to be a tuple $(X_A,i,p,E_A,\alpha)$ where $(X_A,i,p)$ is a deformation of $X$ over $A$, $E_A$ is a vector bundle over $X_A$ and $\alpha:i^*E_A\to E$ is an isomorphism.
An isomorphism $(\phi,\psi):(X_A,i,p,E_A,\alpha)\to(X'_A,i',p',E'_A,\alpha')$ is a pair of an isomorphism $\phi:(X_A,i,p)\to(X'_A,i',p')$ and an isomorphism $\psi:\phi^*E'_A\to E_A$ such that $\alpha'=\alpha\circ i^*\psi$. The deformation functor $\Def(X,E)$ is defined by $\Def(X,E)(A):=$ the set of isomorphism classes of deformations of $(X,E)$ over $A$, and $T^i\Def(X,E)=\Ext^i(\cP_E,\cO_X)$ where $\cP_E\in D^-(X)$ is defined as the mapping cone of the Atiyah class $\at_E$.
Note that if $X$ is lci, then $\cP_E$ is perfect in $[-1,0]$.

We define a deformation of $(X,E,s_1,\ldots,s_n)$ over $A\in \Art$ to be a tuple $(X_A,i,p,E_A,\alpha,t_1,\ldots,t_n)$ where $(X_A,i,p,E_A,\alpha)$ is a deformation of $(X,E)$ over $A$, and $t_i$ are global sections of $E_A$ over $X_A$ such that, for $j=1,\ldots,n$, we have $\alpha^*i^*t_j=s_j$.
An isomorphism $$(\phi,\psi):(X_A,i,p,E_A,\alpha,t_1,\ldots,t_n)\to(X'_A,i',p',E'_A,\alpha',t'_1,\ldots,t'_n)$$ is an isomorphism $(\phi,\psi):(X_A,i,p,E_A,\alpha)\to(X'_A,i',p',E'_A,\alpha')$ such that, for $j=1,\ldots,n$, we have $\psi^*t_j'=t_j$. The deformation functor $\Def(X,E,s_1,\ldots,s_n)$ is defined by $\Def(X,E,s_1,\ldots,s_n)(A):=$ the set of isomorphism classes of deformations of $(X,E,s_1,\ldots,s_n)$ over $A$.

 A natural transformation $F\to G$ of deformation functors is {\em smooth} if, for every surjection $A\to B$ in $\Art$, the induced map \[
 F(A)\to G(A)\times_{G(B)}F(B)
 \]
is surjective. A morphism of schemes or DM stacks $f:X\to Y$ is smooth at $x\in X$ if and only if the induced natural transformation  $h_{X,x}\to h_{Y,f(x)}$ is smooth.

Note that if $F$ is a deformation functor and $T^2F=0$, then $F\to \pt$ is smooth. In particular, if $F=h_{X,x}$ and $T^2F=0$, then $X$ is smooth at $x$.

The forgetful morphism $\Def(X,E,s_1,\ldots,s_n)\to \Def(X,E)$ is smooth if $H^1(X,E)=0$.


\section{Proof of the Main Theorem and variations}\label{Proof_of_main_thm}

\subsection{Proof of the main theorem}

In this subsection we assume throughout the assumptions and notation of Theorem \ref{main_theorem}.

\begin{lemma}\label{key_lemma} By replacing $M$ with a dense open subset containing $m_0$, we can assume that: \begin{enumerate}
    \item $M$ is smooth;
    \item the morphism $f$ is lci, hence both $X_M$ and each fiber $X_{m}:=\pi^{-1}(m)$ have lci singularities;
    \item the total space $X_M$ of the family is smooth.
\end{enumerate}
\end{lemma}

\begin{proof}
(1)
  The locus of smooth points is an open dense subset.

    (2)     Since $\pi$ is flat, the locus in $X_M$  of points $p$ where $\pi $ is lci at $p$ is open in $X_M$, see for instance \cite[\S 5]{N}. On the other hand,  since $\pi$ is proper, the locus of points $m$ in $M$ where $\pi$ is lci at every point of $X_m$ is open.

 (3)  We consider the distinguished triangle of perfect complexes \[
    \pi^*L_M\to L_{X_M} \to L_\pi\to \pi^*L_M[1].
    \]
    Its dual is also perfect and commutes with base change. Taking the long exact sequence in cohomology gives an exact sequence of coherent sheaves on $X_M$: \[
    \pi^*T_M \to T^1_\pi \to T^1_{X_M}\to 0=h^1(\pi^*L_M^\vee).
    \]
    Its restriction to $X_m$ becomes an exact sequence of coherent sheaves on $X_m$:
    \[
    T_mM\otimes \cO_{X_m}\to T^1_{X_m}\to T^1_{X_M}|_{X_m}\to 0.
    \]
    By assumption the first map is surjective on $X_{m_0}$, thus  $T^1_{X_M}|_{X_{m_0}}=0$. Therefore the support of $T^1_{X_M}$ is disjoint from $X_{m_0}$; it is also closed, so its image in $M$ is closed and does not contain $m_0$. We can replace $M$ by the complement of this image, with the result that $T^1_{X_M}=0$ which in turn means that $X_M$ is smooth, since for every $x\in X_M$, the obstruction space
    $T^2h_{{X_M},x}=T^1_{X_M}(x):=T^1_{{X_M},x}\otimes \kappa(x)$ to $h_{X_M,x}$ vanishes.
\end{proof}

Note that $\pi$ is necessarily surjective. Indeed,  $\pi(X_M)$ is closed because $\pi$ is proper, and  $\pi$ is open because $\pi$ is flat. Thus $\pi(X_M)$ is open in $M$. Since $M$ is irreducible, $\pi$ must be surjective.
\begin{proof}[Proof of Theorem \ref{main_theorem}]
     By replacing $M$ with a dense open subset containing $m_0$, we may assume that $X_M$ is smooth. This proves the theorem by applying Generic Smoothness \cite[Corollary III.10.7]{Ha}, which guarantees that for any morphism $\pi:X_M \to M$ of smooth varieties there exists an open nonempty subset $V\subset M$ such that $\pi:\pi^{-1}(V)\to V$ is smooth.
\end{proof}


\subsection{Variations}

\begin{proposition}\label{main_thm_stacks}
    Theorem \ref{main_theorem} also holds if $M$ and $X_M$ (and hence also $X$) are assumed to be Deligne Mumford stacks instead of schemes.
\end{proposition}
\begin{proof}
    As a first step, we reduce to $M$ being a scheme by passing to a local \'etale cover of $M$ near $m_0$ which is a scheme. This does not change the fibers and does not change the smoothness at $m_0$ condition.

     The rest of the proof of the proposition is the same, hence after restriction we may assume $X_M$ to be smooth. By definition of DM stack, we can find a morphism $g:U\to X_M$ such that $U$ is a scheme of finite type over $\K$, $g$ is representable, surjective and \'etale; we denote by $\tilde \pi:=\pi\circ g$. Thus $U$ is a disjoint union of finitely many smooth varieties $U_k$. For each $k$, generic smoothness gives us a nonempty open $V_k$ in $M$ such that $\tilde\pi:\tilde\pi^{-1}(V_k)\cap U_k\to V_k$ is smooth. Let $V$ be the intersection of the $V_k$, which is nonempty because $M$ is irreducible. Then $\tilde\pi:\tilde\pi^{-1}(V)\to V$ is smooth. Since $g:\tilde\pi^{-1}(V)\to \pi^{-1}(V)$ is \'etale and surjective, it follows that $\pi:\pi^{-1}(V)\to V$ is smooth.
\end{proof}

     We can also prove a weaker version of Theorem \ref{main_theorem} by showing that the dimension of the support of $\coker (T_{m_0}M\otimes \cO_X\to T^1_X)$ estimates the dimension of the singular locus of the general fiber.

\begin{proposition}\label{singularity_estimate}
    Given a flat proper family $\pi:X_M\to M$ with $M$ irreducible, if there is $m_0\in M$ such that \begin{enumerate}
    \item $M$ is smooth at $m_0$,
    \item the fiber $X:=X_{m_0}:=\pi^{-1}(m_0)$ has lci singularities,
    \item the support of $\coker (T_{m_0}M\otimes \cO_X\to T^1_X)$ has dimension $d$;
    \end{enumerate}
 then the general fiber $X_m$ is lci and singular in dimension at most $d$.
\end{proposition}
\begin{proof}
    The proof of (1) and (2) in Lemma \ref{key_lemma} is still valid. Let $Z$ be the singular locus of $X_M$. By restricting $M$, we can assume that each irreducible component of $Z$ surjects on $M$. Let $Z_1,\ldots ,Z_k$ be the irreducible components of $Z$; we know that $\dim Z_a\cap X_{m_0}\le d$. This implies that $\dim Z_a\le d+\dim M$. By restricting $M$, we may assume that $\dim (Z_a\cap X_m)\le d$ for every $m\in M$. Write $X_M':=X_M\setminus Z$; then we can apply generic smoothness to $\pi:X_M'\to M$ (it does not have properness as assumption) and find a dense open subset $V\subset M$  such that $\pi:\pi^{-1}(V)\cap X_M'\to V$ is smooth. Then for every $m\in V$, the singular locus of $X_m$ is contained in $Z\cap X_m$ and thus has at most dimension $d$.
\end{proof}


\section{Geometric smoothability criteria}\label{Geometric_smoothability_criteria}

In this section, as application of our main result, we will get sufficient conditions for the
existence of a geometric smoothing for proper lci schemes and DM
stacks,  and for polarized projective lci schemes.

\begin{theorem}\label{smoothable_Fano_gentype}
    Let $X$ be a lci projective scheme such that $\Def_X$ is unobstructed and $T^1\Def_X\otimes \cO_X\to T^1X$ is surjective. If either $\omega _X$ is ample  or $\omega_X^{\vee}$  is ample or $H^2(X,\cO_X)=0$, then $X$ admits a geometric projective smoothing.
\end{theorem}

\begin{proof}

    We fix a line bundle $L$ on $X$ as follows: if $\omega_X$ is ample, we let $L:=\omega_X$; if  $\omega_X^\vee$ is ample, we let $L:=\omega_X^\vee$; if $H^2(X,\cO_X)=0$, we choose $L$ any ample line bundle on $X$. We fix a positive integer $n$ such that $L^{\otimes n}$ is very ample and all $H^i(X,L^{\otimes n})$ vanish for all $i>0$. We write $\cO_X(1):=L^{\otimes n}$ and set $N:=h^0(\cO_X(1))-1=\chi(\cO_X(1))-1$. We let $P$ be the associated Hilbert polynomial, defined by requiring $P(m):=\chi(\cO_X(m))$.

 We fix a projective basis $(v_0,\ldots,v_N)$ of $H^0(\cO_X(1))$. It defines a closed embedding of $X$ in $\PP^N$ with Hilbert polynomial $P$, that is a point $w_0$ in $W:=\Hilb^P(\PP^n)$. The completion $\widehat\cO_{W,w_0}$ represents the deformation functor $\Def(X,\cO_X(1),v_0,\ldots,v_N).$ We note that the forgetful natural transformation $$\Def(X,\cO_X(1),v_0,\ldots,v_N)\to \Def(X,\cO_X(1))$$ is smooth because $H^1(\cO_X(1))=0$.

 We do first the case where $H^2(\cO_X)=0$. Then the natural transformation $\Def(X,\cO_X(1))\to \Def_X$ is smooth as its obstructions space vanishes. Therefore,  $W$ is smooth at $w_0$.
 To apply Theorem \ref{main_theorem} we need to show that $T_{w_0}W\otimes\cO_X\to T^1_X$ is surjective. By assumption, we know that   \(T^1\Def_X\otimes \cO_X\to T^1_X\)
is surjective. Since $$\Def(X,\cO_X(1),v_0,\ldots,v_N)\to \Def(X)$$ is smooth, it induces a surjection on tangent spaces $T_{w_0}W\to T^1\Def_X$ which completes the argument.

 Therefore,  there is only one irreducible component $M$ of $W$ which contains $w_0$. We write $\pi:X_M\to M$ for the universal family on $W$ restricted to $M$.

 By Theorem \ref{main_theorem} there is an open dense subset $V$ in $M$ such that $\pi:\pi^{-1}(V)\to V$ is smooth, and a geometric smoothing of $X$ can be constructed by choosing a morphism $\phi:C\to M$ where $C$ is a smooth irreducible curve, $c_0\in C$ a closed point mapping to $w_0$, such that $\phi$ maps the generic point of $C$ to $V$. The smoothing is then given by $X_C=X_M\times_MC\to C$.

We will next do the case where $\omega_X$ is ample. The case where $\omega^\vee_X$ is ample can be obtained by switching $\omega_X$ with $\omega_X^\vee$ in the following argument.
We consider the locus $W_{lci}$ in $W$ parametrizing lci schemes (i.e., where the morphism from the universal family is lci); $W_{lci}$ is open in $W$ and  it contains $w_0$. Since lci schemes are Gorenstein, we can construct the locus $Z$ in $W_{lci}$ parametrizing $w$ such that $\omega_{X_w}^{\otimes n}\cong \cO_{X_w}(1)$; since the relative Picard scheme is separated, $Z$ is a closed subscheme of $W_{lci}$ containing $w_0$.

In analogy to the previous case, the completion $\widehat\cO_{Z,w_0}$ represents the functor $\Def(X,v_0,\ldots,v_N)$ where $(v_i)$ is a projective basis of $H^0(X, \omega_X^{\otimes n})$. We note that every infinitesimal deformation of $X$ is again lci, and hence it has a relative dualizing sheaf which canonically extends $\omega_X$ to an invertible sheaf on the deformation. Since we assumed that $H^1(\omega_X^{\otimes n})=0$, the forgetful natural transformation $\Def(X,v_0,\ldots,v_N)\to \Def_X$ is smooth,  and therefore $Z$ is nonsingular at $w_0$. We choose $M$ to be the unique irreducible component of $Z$ containing $w_0$ and we conclude arguing as in the previous case.
\end{proof}

\begin{proof}[Proof of Theorem \ref{smoothability_criteria}]
   The assumptions imply, by the local to global spectral sequence of $\Ext$, that $\Def_X$ is unobstructed. Indeed, the obstruction space is $\Ext^2(L_X,\cO_X)$ and both $H^2(T_X)$ and $H^1(X,T^1_X)$ are assumed to be zero; and  there is no $T^2_X$ since $X$ has lci singularities. The vanishing of $H^2(T_X)$ also implies that the morphism $T^1\Def_X:=\Ext^1(L_X,\cO_X)\to H^0(T^1_X)$ is surjective, which, coupled with $T^1_X$ being globally generated, means that \[
T^1\Def_X\otimes \cO_X\to T^1_X
\]
is surjective. The result then follows from Theorem \ref{smoothable_Fano_gentype}.
\end{proof}


\subsection{Smoothability for polarized  schemes} In this subsection, we study the smoothability of polarized  schemes with lci singularities. In the next, we will apply it to strengthen the previous result of the authors about the smoothability of polarized K3 surfaces.

\begin{theorem}\label{smoothing_polarized} Let $X$ be a projective scheme with lci singularities and $L$ an ample line bundle on $X$. Assume that the pair $(X,L)$ has unobstructed deformations, and that $T^1\Def_{(X,L)}\otimes\cO_X\to T^1X$ is surjective. Then $(X,L)$ admits a geometric smoothing.
\end{theorem}

\begin{proof}
 The proof proceeds in the same way as the proof of Theorem  \ref{smoothable_Fano_gentype} in the case $H^2(X,\cO_X)=0$, except now we show that $W$ is smooth at $w_0$ by considering the smooth forgetful morphism $$\Def(X,\cO_X(1),v_0,\ldots,v_N)\to \Def(X,\cO_X(1))$$
    and composing it with the isomorphism (it is an isomorphism because we are in characteristic zero) $$\Def(X,L)\to \Def(X,L^{\otimes n})=\Def(X,\cO_X(1)).$$
\end{proof}

Note that $L\otimes L^\vee$ is canonically isomorphic to $\cO_X$, hence
    the distinguished triangle \[
    \cO_X[-1]\to L_X\to \cP_L\to \cO_X
    \]
    induces
    a long exact sequence
    \[
    H^1(\cO_X)\to \Ext^1(\cP_L,\cO_X)\to \Ext^1(L_X,\cO_X)\to H^2(\cO_X)\to\]
    \[\to \Ext^2(\cP_L,\cO_X)\to \Ext^2(L_X,\cO_X)\to H^3(\cO_X).
    \]

A sufficient condition for the deformation functor $\Def(X,L)$ to be unobstructed is that $\Ext^2(\cP_L,\cO_X)=0$. This condition is equivalent to the surjectivity of  $\Ext^1(L_X,\cO_X)\to H^2(\cO_X)$ and the injectivity of  $\Ext^2(L_X,\cO_X)\to H^3(\cO_X)$.

\subsection{Smoothing polarized singular K3 surfaces}\label{K3_polar_smoothing}

In \cite{FM}, we define a connected projective Gorenstein scheme $X$ of dimension $2$ to be a {\em singular K3 surface} if $\omega_X=\cO_X$ and $H^1(\cO_X)=0$. In the rest of the subsection we will assume that $X$ is a singular K3 surface with lci singularities.

In \cite{FM} we discussed the smoothability  for pairs $(X,L)$ where $X$ is a lci K3 surface  with isolated singularities and $L$ is an ample line bundle and we give a sufficient criterion in order to assure the geometric smoothability of the pair $(X,L)$ \cite[Theorem 5.12]{FM}. In this subsection we want to find more general conditions for polarized smoothability.

\begin{remark}
    We have an exact sequence \[
    0\to \Ext^1(\cP_L,\cO_X)\to \Ext^1(L_X,\cO_X)\to
    H^2(\cO_X)\]\[\to \Ext^2(\cP_L,\cO_X)\to \Ext^2(L_X,\cO_X)\to 0.
    \]
\end{remark}

\begin{remark}\label{rmk}
     Let $X$ be a singular K3 surface with lci singularities, and $L$ a line bundle on $X$. Since the natural map \[
     \Ext^1(L_X,\cO_X)\times \Ext^1(\cO_X,L_X)\to H^2(\cO_X)
     \]
     is a perfect pairing by Serre duality, in analogy with the argument in  \cite[Lemma 5.10]{FM}; it follows that the map $$\alpha_L:\Ext^1(L_X,\cO_X)\to H^2(\cO_X)=\K$$ induced by the Atiyah class $\at_L$ is nonzero (and hence surjective) if and only if $\at_L\neq 0$. In particular it is nonzero if $L$ is ample.
     Note that this implies that there exist first order deformations of $X$ to which $L$ does not extend, and that $\Ext^2(\cP_L,\cO_X)\to \Ext^2(L_X,\cO_X)$ is an isomorphism.
\end{remark}
 \begin{proposition}\label{smoothing_K3}
     Let $X$ be a singular K3 surface with lci singularities. The assumptions of Theorem \ref{smoothing_polarized} are verified if  \begin{enumerate}
       \item $H^1(X,T^1_X)=0$;
       \item $H^0(T^1_X)\to H^2(X,T_X)$ is surjective;
       \item $T^1_X$ is generated by its global sections in the image in $H^0(T^1_X)$ of $\ker \alpha_L\subset \Ext^1(L_X,\cO_X)$.
        \end{enumerate}
 \end{proposition}
  \begin{proof}
      (1) and (2) together imply $\Ext^2(L_X,\cO_X)=0$, and by Remark \ref{rmk} the deformation functor $\Def(X,L)$ is unobstructed.
      Finally, condition (3) states that $T^1\Def(X,L)\otimes \cO_X\to T^1_X$ is surjective.
  \end{proof}

  \begin{remark}
       Notice that the hypothesis (1) and the first half of hypothesis (3) are automatically satisfied when we deal with isolated singularities.
  \end{remark}

\vskip 4mm

\begin{remark} In  \cite[Example  5.16]{FM} the authors construct an example of lci K3 surface with isolated singularities which is polarized smoothable, but to which our smoothability criterion in that paper does not apply: this was part of the motivation to seek broader criteria.

We will now check that the assumptions of Proposition \ref{smoothing_K3} apply in this case. By the remark, we only need to check (2) and the second half of (3). (2) follows by $\Ext^2(\Omega_X,\cO_X)=0$ which was proved for all quartics in $\mathbb{P}^3$ in   \cite[Lemma 5.14]{FM}.

Finally, $T^1_X=\cO_X(4)|_Z$ where $Z\subset X$ is the scheme theoretic singular locus, and the smooth morphism $\Hilb(\mathbb{P}^3)\to \Def(X,L)$ induces a surjection on the differentials $H^0(X,\cO_X(4))\to T^1\Def(X,L)$. This reduces the claim to showing that $|\cO_X(4)|$ is base point free.
\end{remark}

\begin{lemma}
    Assume that $X$ has isolated rational singularities and that $\Omega_X$ is torsion-free. Then $\Ext^2(L_X,\cO_X)=0$.
\end{lemma}
\begin{proof}
  We assume $X$ is reduced. Hence, we have $\Ext^2(L_X,\cO_X)=\Ext^2(\Omega_X,\cO_X)=H^0(\Omega_X)^\vee$. Let $\eps:\tilde X \to X$ be a resolution of singularities. Then $H^0(\Omega_X)\neq 0$ implies that $H^0(\Omega_{\tilde X})\neq 0$, thus $H^1(\cO_{\tilde X})\neq 0$. On $X$ we have an exact sequence \[
 0\to \cO_X \to \eps_*\cO_{\tilde X}\to G\to 0
 \] where $G$ is supported on the singular locus, hence $H^1(G)=0$. The assumption that $X$ has rational singularities is equivalent to $R^1\eps_*\cO_{\tilde X}=0$. By the Leray spectral sequence, $H^1(\eps_*\cO_{\tilde X})\to H^1(\cO_{\tilde X})$ is an isomorphism, therefore $H^1(\cO_X)\neq 0$, against our assumption.
\end{proof}

\section{Smoothability of stable varieties}\label{smoothing_stable_surfaces}

\subsection{On the moduli stack of stable varieties}

The moduli stack  of smooth complex projective varieties of fixed dimension $d$ with  ample dualizing sheaf and fixed numerical invariants is of finite type and its coarse moduli space is a quasiprojective scheme. For $d=1$, it was compactified by Deligne and Mumford as stack of stable curves; each connected component $\overline{\cM}_g$ is smooth and irreducible, and in particular each stable curve is smoothable.

For $d=2$, Koll\'ar and Shepherd-Barron \cite{KSB} first proposed a modular compactification with projective coarse moduli space, parametrizing so called stable surfaces; the missing boundedness was proven by Alexeev in \cite{A}.

This construction was then generalized to  any $d$ (and also pairs of varieties with $\QQ$ divisors,  and more): we refer the reader to \cite{Ko} for a comprehensive presentation.

\begin{Notation} We will fix from here on $d\ge 2$ and suitable numerical invariants, and denote by $\cM$ the moduli stack of smooth projective varieties of dimension $d$ with ample dualizing sheaf and fixed numerical invariants; we will denote by $\overline\cM$ the stack of stable varieties with the same dimension and invariants.
\end{Notation}

\begin{remark}
Defining a stack means defining families. If $B$ is a reduced scheme, a morphism $B\to \overline{\cM}$ is a family of stable varieties, that is a flat proper morphism $S_B\to B$ whose fibers are stable varieties. If $B$ is non-reduced one must additionally require that the family is $\QQ$-Gorenstein, in an appropriate technical sense.
\end{remark}

Note in particular that, for $S$ a stable non Gorenstein variety, the usual deformation functor $\Def_S$ has a subfunctor  $\Def^{QG}_S$ of $\QQ$-Gorenstein deformations, which is canonically isomorphic to $h_{\overline{\cM},[S]}$.

We will take advantage of an alternative definition of $\overline\cM$ due to  Abramovich and Hassett, who in \cite{AH} showed that one could interpret $\overline{\cM}$ as moduli stack of (flat proper families whose fibers are) Gorenstein DM stacks $X$, each of which is the index-one cover of its coarse moduli space $S$, which is a stable variety of fixed dimension and invariants. In particular $\Def_X$ is canonically isomorphic to $h_{\overline{\cM},[S]}$.

We recall here that the index-one cover stack of a stable variety $S$ is a morphism $\eps:X\to S$, where $X$ is a DM stack, such that, for every
affine open $U$ in $S$ and integer $r>0$ such that such that $rK_U\sim \cO_U$, $\eps^{-1}(U)$ is isomorphic over $U$ to the stack quotient $[V/\mu_r]$, where  $V\to U$ is the canonical induced $\mu_r$ Galois cover with $V$ Gorenstein, branched exactly over the non Gorenstein points (which are in codimension $\ge 2$).

In particular, if $S$ is a stable Gorenstein variety, then its index-one cover stack is $\id_S:S\to S$.

\subsection{Smoothability for stable varieties}

\begin{theorem}\label{thm_smoothing_stable_surfaces}
    Let $S$ be a stable variety and $X$ the associated DM stack. Assume that $X$ has lci singularities. If $[S]$ --- or, equivalently, $[X]$ --- is a smooth point of $\overline{\cM}$ and $T^1\Def_X\otimes \cO_X\to T^1_X$ is surjective, then $[S]$ is geometrically smoothable and in the closure of $\cM$.
\end{theorem}
\begin{proof}
    We apply Proposition \ref{main_thm_stacks} to the universal family $\cX\to \overline{\cM}$ restricted to the open nonsingular (and nonempty) locus $U$ in the unique irreducible component of $\overline{\cM}$ containing $[S]$.
\end{proof}

\begin{corollary}\label{stable_smoothing} Let $S$ be a stable variety  and $X$ the associated DM stack. Assume that $X$ has lci singularities. If $H^1(T^1_X)=H^2(T_X)=0$ and $T^1_X$ is globally generated, then $[S]$ is in the closure of $\cM$.
\end{corollary}
\begin{proof}
By the local to global spectral sequence of $\Ext$, the assumptions guarantee that the functor $\Def_X$ is smooth (since its obstruction space $\Ext^2(L_X,\cO_X)=0$) and that $T^1\Def_X\to H^0(T^1_X)$ is surjective; together with the assumption that $T^1_X$ is generated by global sections, this implies that the assumptions of Theorem \ref{thm_smoothing_stable_surfaces} are satisfied.
\end{proof}

Since $\eps$ is proper and bijective on points, and $\eps_*T_X\to T_S$ is an isomorphism, the condition $H^2(T_X)=0$ can be restated as $H^2(T_S)=0$; similarly, the conditions on $T^1_X$ can be stated for $T^{1,QG}_S:=\eps_*T^1_X$ instead.

In the case where $S$ is Gorenstein (i.e., $X=S$) Corollary \ref{stable_smoothing} is already proven by Nobile \cite[Theorem 5.7]{N}, again as a consequence of Tziolas' work \cite[Theorem 12.5]{Tz}.

Corollary \ref{stable_smoothing} can be proven in full generality by combining Tziolas' result on formal smoothability of $\QQ$-Gorenstein varieties (the second part of \cite[Theorem 12.5]{Tz})  with an extension of Nobile's characterization of formal smoothings \cite[Theorem 5.11]{N} to DM stacks.

\end{document}